\documentclass[a4paper,10pt,reqno]{amsart}
\usepackage[a4paper,tmargin=40mm,bmargin=35mm,hmargin=30mm]{geometry}
\usepackage[T1]{fontenc}
\usepackage{dsfont}
\usepackage{graphicx}
\usepackage{latexsym}
\usepackage[all]{xy}
\usepackage{lmodern}
\usepackage{amsmath}
\usepackage{amssymb}
\usepackage{amsthm}
\usepackage{stmaryrd}
\usepackage{amscd}
\usepackage{tikz}
\usepackage{tikz-cd}
\usepackage{booktabs}
\usepackage{multirow}
\usepackage{multicol}
\usepackage{here}
\usepackage{aliascnt}

\usepackage{hyperref}
\hypersetup{%
  bookmarksnumbered=true,%
  colorlinks=true,%
  linkcolor=blue,%
  citecolor=blue,%
  filecolor=blue,%
  menucolor=blue,%
  urlcolor=blue,%
  bookmarksopen=true,%
  bookmarksdepth=2,%
  pageanchor=true}
\usepackage[noabbrev]{cleveref}

\newcommand{\NewTheorem}[2]{
	\newaliascnt{#1}{TheoremEnvironment}
	\newtheorem{#1}[#1]{#1}
	\aliascntresetthe{#1}
	\crefname{#1}{#1}{#2}
	\Crefname{#1}{#1}{#2}
}

\newcommand{\longiso}{\xrightarrow{\ \raisebox{-.4ex}[0ex][0ex]{$\scriptstyle{\sim}$}\ }}
\theoremstyle{definition}

\NewTheorem{Definition}{Definitions}

\theoremstyle{remark}
\NewTheorem{Remark}{Remarks}
\NewTheorem{Axiom}{Axioms}
\NewTheorem{Example}{Examples}

\NewTheorem{Observation}{Observations}
\NewTheorem{Convention}{Conventions}
\NewTheorem{Notation}{Notations}
\NewTheorem{Setting}{Settings}
\NewTheorem{Question}{Questions}
\NewTheorem{Answer}{Answers}
\NewTheorem{Conjecture}{Conjectures}
\NewTheorem{Problem}{Problems}
\NewTheorem{Solution}{Solutions}
\NewTheorem{Definition and Remark}{Definition and Remark}
\NewTheorem{Comment}{Comments}
\NewTheorem{Aim}{Aims}
\NewTheorem{Caution}{Cautions}
\NewTheorem{Exercise}{Exercises}

\theoremstyle{plain}
\NewTheorem{Proposition}{Propositions}
\NewTheorem{Lemma}{Lemmas}
\NewTheorem{Theorem}{Theorems}
\NewTheorem{Corollary}{Corollaries}

\crefname{enumi}{}{}
\Crefname{enumi}{}{}
\creflabelformat{enumi}{(#2#1#3)}
\crefname{enumii}{}{}
\Crefname{enumii}{}{}
\creflabelformat{enumii}{(#2#1#3)}
\crefname{enumiii}{}{}
\Crefname{enumiii}{}{}
\creflabelformat{enumiii}{(#2#1#3)}

\makeatletter
\renewcommand{\p@enumii}{}
\renewcommand{\p@enumiii}{}
\makeatother

\numberwithin{equation}{section}
\crefname{equation}{}{}
\Crefname{equation}{}{}
\creflabelformat{equation}{(#2#1#3)}

\newcommand{\SwapSymbols}[1]{
	\expandafter\let\expandafter\temporarysymbol\csname #1\endcsname
	\expandafter\let\csname #1\expandafter\endcsname\csname var#1\endcsname
	\expandafter\let\csname var#1\endcsname\temporarysymbol
}

\SwapSymbols{epsilon}
\SwapSymbols{phi}
\SwapSymbols{Gamma}
\SwapSymbols{Delta}
\SwapSymbols{Theta}
\SwapSymbols{Lambda}
\SwapSymbols{Xi}
\SwapSymbols{Pi}
\SwapSymbols{Sigma}
\SwapSymbols{Upsilon}
\SwapSymbols{Phi}
\SwapSymbols{Psi}
\SwapSymbols{Omega}

\newcommand{\bL}{\mathbf{L}}

\newcommand{\bbZ}{\mathbb{Z}}
\newcommand{\cA}{\mathcal{A}}

\newcommand{\cC}{\mathcal{C}}

\newcommand{\cE}{\mathcal{E}}
\newcommand{\cF}{\mathcal{F}}

\newcommand{\cI}{\mathcal{I}}

\newcommand{\cL}{\mathcal{L}}

\newcommand{\cO}{\mathcal{O}}

\newcommand{\cS}{\mathcal{S}}
\newcommand{\cT}{\mathcal{T}}
\newcommand{\cU}{\mathcal{U}}
\newcommand{\cV}{\mathcal{V}}

\newcommand{\cX}{\mathcal{X}}

\newcommand{\To}{\rightarrow}

\DeclareMathOperator{\Hom}{Hom}
\DeclareMathOperator{\RHom}{{\bf R}Hom}

\DeclareMathOperator{\Ext}{Ext}

\DeclareMathOperator{\Zg}{Zg}

\DeclareMathOperator{\Mod}{Mod}

\DeclareMathOperator{\ZSupp}{ZSupp}
\DeclareMathOperator{\suppex}{supp_{ex}}
\DeclareMathOperator{\dSpec}{Zspec}

\DeclareMathOperator{\Inj}{Inj}

\DeclareMathOperator{\fp}{fp-}

\DeclareMathOperator{\soc}{soc}
\DeclareMathOperator{\rad}{rad}
\DeclareMathOperator{\SSupp}{Zsupp}

\DeclareMathOperator{\Ann}{Ann}

\DeclareMathOperator{\Ker}{Ker}
\DeclareMathOperator{\Coker}{Coker}
\let\Im\relax
\DeclareMathOperator{\Im}{Im}

\DeclareMathOperator{\Spec}{Spec}

\DeclareMathOperator{\Ass}{Ass}
\DeclareMathOperator{\Supp}{supp}

\DeclareMathOperator{\Sp}{Sp}

\DeclareMathOperator{\Lf}{{\bf L}f}

\title{Classifying subcategories of a Grothendieck category via its spectral category}


\author{Reza Sazeedeh}

\address{Department of Mathematics, Urmia University, P.O.Box: 165, Urmia, Iran}
\email{rsazeedeh@ipm.ir and r.sazeedeh@urmia.ac.ir}

\begin{document}

\begin{abstract}
Let $\cA$ be a Grothendieck category. In this paper we classify subcategories of $\cA$ via a support notion defined by Krause \cite{Kr2024} based on the spectral category of $\cA$.  We show that this aligns with Nemman's support \cite{Ne1992}, and thereby extends earlier classifications of subcategories of the category of modules over  a commutative noetherian ring and subcategories of its derived category. Over a locally coherent category $\cA$, we show that  this support agrees with open subsets of the Ziegler topology on $\cA$. 
  
	\end{abstract}

\maketitle


\section{Introduction}

The classification of subcategories of a Grothendieck category $\cA$ has a long history and remains an active area of research to this day (see \cite{St2024}). This began with Gabriel \cite{G1962}, who classified localizing subcategories of the category of modules over a commutative noetherian ring $A$ using specialization-closed subsets of $\Spec A$. Subsequently various subcategories of $\Mod A$ and the derived category ${\bf D}(A)$ has been classified by several authors (see \cite{Kr2008, MT2022,Ne1992,Ta2009}). Over a Grothendieck category $\cA$ a suitable analogous of $\Spec A$ is $\Sp\cA$, the set of  isomorphism classes of indecomposable injective objects in $\cA$.
For a locally coherent category, Herzog \cite{He1997} classified  localizing subcategories of finite type of $\cA$ via a collection of open subsets of $\Sp\cA$, originally introduced by Ziegler \cite{Zi}.

Another option is  the spectral category $\Spec \cA$ introduced  by Gabriel and Oberst \cite{GO}. Its advantage is that $\Spec \cA$ is itself a Grothendieck category, allowing one to work with its subcategories rather than  with open subsets of a topology. 

Recently, Krause \cite{Kr2024} has studied ${\bf L}(\Spec\cA)$, the lattice of localizing subcategories of $\Spec\cA$. He  assigned to each $X\in\cA$, its exact support $\suppex(X)\in{\bf L}(\Spec\cA)$ which is a suitable replacement for support of a module defined by Neeman \cite{Ne1992} when one  consider it as a complex in the derived category of a commutative noetherian ring. Our main aim of this paper is to show that the Karuse's support can extend the classifications of certain subcategories of $\Mod(A)$ and ${\bf D}(A)$ of a commutative noetherian ring $A$ to Grothendieck category $\cA$ and ${\bf D}(\cA)$. 

There is a canonical functor $P:\cA\to\Spec\cA$ which identifies $\Sp\cA$ with the isomorphism classes of simple objects in $\Spec\cA$. Let $X$ be an object of $\cA$ and $X\to \cI$ be a minimal injective resolution. Then
$\suppex(X)$ equals the localizing subcategory of $\Spec\cA$ consisting of all direct
summands of coproducts of objects in $\{P(I^n)\mid n\in\bbZ\}$. As every complex in ${\bf D}^+(\cA)$ admits a minimal injective resolution  (\cite[Proposition B2]{Kr2005}), the exact support of a complex can defined analogously. 
 
 A subcategory $\Phi\in{\bf L}(\Spec\cA)$ is said to be {\it coherent} if each morphism $I^0\To I^1$ of injective objects of $\cA$ with $P(I^0), P(I^1)
 \in\Phi$ fits into an exact sequence $I^0\To I^1\To I^2$ with $P(I^2)\in\Phi$.
 
In Section 2, we prove the following theorem which extends \cite[Theorem 5.2]{Kr2008}.

\begin{Theorem}[\cref{krau}]
For $\Phi\in{\bf L}(\Spec\cA)$, the following conditions are equivalent.
\begin{enumerate}
\item $\Phi$ is coherent.
\item For every complex $X$ in ${\bf D}^+(\cA)$ we have:
\[\suppex X\subseteq \Phi \hspace{0.2cm}\Longleftrightarrow\hspace{0.2cm} \suppex H^i(X)\subseteq \Phi\text{ for all }i.\]

\item For every complex $X$ in ${\bf D}^+(\cA)$ we have:
\[\suppex(X)\subseteq \Phi \hspace{0.2cm}\Longrightarrow\hspace{0.2cm} \suppex(H^i(X))\subseteq \Phi\text{ for all }i.\]
\end{enumerate}
 \end{Theorem}

We also prove the following theorem which extends \cite[Theorem 2.12]{Ta2009}.

 \begin{Theorem}[\cref{loc}]
 The assignment $\cL\longmapsto \suppex(\cL\cap \cA)$ induces an inclusion preserving between   \[\{\text{localizing subcategories }\cL\text{ of }{\bf D}^{+}(\cA) \text{ such that }\cL \text{ is closed under homology and }\]\[\cL\cap\cA \text{ is closed under injective envelopes} \}\]\[\longiso\{\text{ coherent subcategories  of }\Spec\cA\}.\]
The inverse map is given by $\Phi\longmapsto \suppex^{-1}(\Phi)$.
\end{Theorem}

A full subcategory $\cT$ of $\cA$ is called {\it thick} if  it is closed under kernels, cokernels and extensions.  The basic properties of the functor $P$ implies that for every $\Phi\in\bL(\Spec\cA)$, the subcategory $P^{-1}(\Phi)$ is closed under injective envelopes so that $\cL=^\bot P^{-1}(\Phi)$ is a localizing subcategory of $\cA$. 

We show that  if $P^{-1}(\Phi)$ is closed under products, then  $\Phi$ is  coherent if and only if  $\cL^{\bot_{\leq 1}}$ is a thick subcategory of $\cA$ (\cref{charc}).

 A subcategory $\Phi\in{\bf L}(\Spec\cA)$ is said to be {\it specialization-closed} if for each non-zero morphism $I^0\To I^1$ of injective objects in $\cA$ such that $\Im(I^0\To I^1)$ is an essential subobject of $I^1$, the condition $P(I^0)\in\Phi$ implies that $P(I^1)\in{\Phi}$. A localizing subcategory is called
\emph{stable} if it is closed under essential extensions. We extend a classification of localizing subcategories established by Garbriel \cite{G1962} for Grothendieck categories.  

\begin{Theorem}[\cref{stablesub}]
 The assignment $\cC\longmapsto P(\cC)$ induces a lattice isomorphism 
\[\{\text{stable localizing subcategories of } \cA\}\longiso\{\text{specialization-closed 
subcategories of } \Spec\cA\}.\]
The inverse map is given by $\Phi\longmapsto P^{-1}(\Phi)$.
\end{Theorem}
 In Section 3, we assume that $\cA$ is locally coherent. A full subcategory of $\cA$ is called \emph{essentially closed} if it
  is closed under arbitrary coproducts, subobjects, and essential
  extensions. For a class $\cX$ of $\cA$ let
$\langle \cX\rangle$ denote the smallest essentially closed subcategory of $\cA$  containing $\cX$. Krause \cite{Kr2024} defined the support of $\cX$ as $\langle P(\cX)\rangle$. In $\Spec\cA$ essentially closed subcategories are precisely localizing subcategories.  For an arbitrary subcategory $\cC$ of $\fp\cA$, {\it the Ziegler support of} $\cC$ is defined as $\SSupp(\cC)=\langle P(\cO(\cC)\rangle$, where $$\cO(\cC)=\{E\in\Sp\cA\mid \Hom(M,E)\neq 0\text{ for some } M\in\cC\}$$ is an open subset of $\Sp\cA$ in Ziegler topology. We show that the Ziegler support and the support of localizing subcategories of finite type of $\cA$ defined by Krause \cite{Kr2024} are closely related to open subsets in the Ziegler topology of $\cA$ (\cref{lft}). To be more precise, if $\cC$ is a localizing subcategory of finite type of $\cA$, then 
 $\Supp(\cC)_d=\SSupp (\cC\cap\fp\cA)$ and 
 $P^{-1}(\Supp(\cC))\cap \Sp\cA=\cO(\cC\cap\fp\cA)$, where $\Supp(\cC)_d$ is the dicrete component of $\Supp(\cC)$. Set
\begin{itemize}
\item $\Lf(\cA)=$ the lattice of localizing subcategories of finite type of $\cA$ ;
 
\item ${\bf L}(\fp\cA)=$  the lattice of Serre subcategories of $\fp\cA$;

\item $\cE(\Sp\cA)=\{\langle\cO(\cC)\rangle\mid \cC \text{ is a subcategory of }\fp\cA\}$;

\item  $\bL(\dSpec\cA)=\{\SSupp(\cC)\mid \cC \text{ is a subcategory of }\fp\cA\}$;

 \end{itemize}
 
 The following theorem provide some classification for localizing subcategories of finite type of $\cA$.  
\begin{Theorem}[\cref{char1,coch}] 
There are the following lattice isomorphisms:

\[\begin{tikzcd} 
\Lf(\cA)  \arrow[rightarrow,yshift=.75ex,rr,"(-)\cap\fp\cA"]
 	&& \arrow[rightarrow,yshift=-.75ex,ll,"\overrightarrow{(-)}"] 
 	\bL(\fp\cA) \arrow[rightarrow,yshift=.75ex,rr,"\SSupp(-)"]
 	&& \arrow[rightarrow,yshift=-.75ex,ll,"\SSupp^{-1}(-)"] 
 	\bL(\dSpec\cA) \arrow[rightarrow,yshift=.75ex,rr,"P^{-1}(-)"]
 	&& \arrow[rightarrow,yshift=-.75ex,ll,"P(-)"] 
 	\cE(\Sp\cA)
\end{tikzcd}.\]

Moreover there is an inclusion-reversing bijection
\[\begin{tikzcd}
\{\Phi\in{\bf L}(\Spec \cA)\mid P^{-1}(\Phi)\text{ is closed under products and filtered colimits}\} \arrow[rightarrow,yshift=1.25ex,rr,"^\bot(P^{-1}(-))"]
 	&& \arrow[rightarrow,yshift=-1.25ex,ll,"P((-)^\bot)"] 
\Lf(\cA)
\end{tikzcd}.\]

\end{Theorem}

In Section 3, we  show that our results and definitions align with the classical ones in locally noetherian categories, specifically in the category of modules over a commutative noetherian ring.

\section{In the case of Grothendieck categories}

Let $\cA$ be a Grothendieck category.  A full subcategory of $\cA$ is
\emph{localizing} if it is closed under subobjects, quotients,
extensions, and arbitrary coproducts. Let $G$ be a generator of $\cA$ and so $\cA$ is locally small. Then any localizing subcategory $\cU$ can be generated by the set of of quotients of $G$. This implies that the localizing subcategories of
$\cA$ form a set. The localizing subcategories are
partially ordered by inclusion and closed under arbitrary
intersections. Then they form a complete lattice which we denote by
${\bf L}(\cA)$. 

Let $\cU\in{\bf L}(\cA)$ and $X\in\cA$. The {\it radical functor} corresponding to $\cU$ is denoted by $t_{\cU}$ where $t_{\cU}(X)$ is the largest subobject of $X$   belonging to $\cU$. The join in ${\bf L}(\cA)$ is defined via the following lemma.

\medskip 
\begin{Lemma}[\cite{Kr2024}, Lemma 2.2]\label{krlem}
 Let $(\cU_\alpha)$ be a family in
${\bf L}(\cA)$ and set $\cU=\bigvee_\alpha\cU_\alpha$. For any $X\in\cA$ we have $t_\cU(X)=\sum_\alpha t_{\cU_\alpha}(X)$. In particular,
$t_\cU(X)=0$ if and only if $t_{\cU\alpha}(X)=0$ for all $\alpha.$
\end{Lemma}

\medskip

\begin{Definition}
A Grothendieck category $\cC$ is a \emph{spectral category} if every short exact sequence in $\cC$ splits. In this case, every object $X\in\cC$ admits a decomposition \[X=\soc(X)\oplus\rad(X),\] where the
\emph{socle} $\soc(X)$ is the sum of all simple subobjects and the
\emph{radical} $\rad(X)$ is the intersection of all maximal
subobjects. Then for any subcategory $\cX$ of $\cC$ we we can define full subcategories 
\[\cX_{\mathrm d}:=\{X\in\cX\mid\soc(X)=X\}\qquad\text{and}\qquad
  \cX_{\mathrm c}:=\{X\in\cX\mid\rad(X)=X\}\]   
This induces a decomposition $\cX=\cX_d\times X_d$ where $\cX_d$ and  $\cX_c$ are called \emph{discrete} and \emph{continuous} components  of $\cX$, respectively. 
For a Grothendieck subcategory $\cA$, we associate to $\cA$ a spectral category $\Spec\cA$. The objects of $\Spec\cA$ is 
the same as the objects of $\cA$ and for objects $C$ an $D$ in $\cA$, we set \[
\Hom_{\Spec\cA}(C,D)=\underset{\rightarrow}{\rm lim}\Hom_\cA(C',D)\]
where the direct limits is taken over a family of essential subobjects $C'$  of $C$. There is a canonical left exact functor $P:\cA\to\Spec\cA$ which leaves objects unchanged and carries each essential monomorphism in $\cA$ into an isomorphism in $\Spec\cA$ (for more details, see \cite{GO,St}).
\end{Definition}

\begin{Definition}
  A full subcategory of $\cA$ is called \emph{essentially closed} if it
  is closed under arbitrary coproducts, subobjects, and essential
  extensions.
\end{Definition}
The essentially closed subcategories are partially ordered by inclusion and
closed under arbitrary intersections; so they form a complete
lattice. A theorem due to Krause \cite[Theorem 3.3]{Kr2024} shows that there is a lattice isomorphism between  essentially closed subcategories of $\cA$ and localizing subcategories of $\Spec\cA$. For a class $\cX$ of $\cA$ let
$\langle \cX\rangle$ denote the smallest essentially closed subcategory of $\cA$  containing $\cX$.

The following definition was given by Krause \cite{Kr2024} for objects in $\cA$. Since any complex in ${\bf D}^+(\cA)$ admits a minimal injective resolution (see \cite[Proposition B.2]{Kr2005}), we define it for complexes in ${\bf D}^+(\cA)$.

\begin{Definition}
For a Grothendieck category $\cA$, consider the derived
  category ${\bf D}(\cA)$ and the right derived functor
\[{\bf R}P\colon{\bf D}(\cA)\longrightarrow{\bf D}(\Spec\cA).\]
For an object $X\in{\bf D}^+(\cA)
$, the \emph{exact support} of $X$ is defined as

 \[ \suppex(X):=\inf\{\cU\in{\bf L}(\cA)\mid H^i({\bf R}
  P(X))\in\cU \text{ for all } i\in\bbZ\}.\]
  
  The function $\suppex$ is {\it exact} as it 
 satisfies the following conditions (see \cite[Lemma 5.2]{Kr2024}):

\begin{enumerate}
  \item for all objects $X,Y\in{\bf D}^+(\cA)$ one has $\suppex(X\oplus Y)=\suppex(X)\vee\suppex(Y)$, and
  \item for each exact triangle  $X_1\to X_2\to X_3\to X_1[1]$ in ${\bf D}^+(\cA)$
    one has
    \[\suppex(X_i)\subseteq\suppex(X_j)\vee\suppex(X_k)\qquad\text{when}\qquad\{i,j,k\}=\{1,2,3\}.\]
  \end{enumerate}
\end{Definition}
The exactness is inspired by the notion of support datum for tensor traingulated categories introduced by Balmer \cite{Ba2005}. For $\cX\subseteq{\bf D}^+(\cA)$, we set $\suppex(\cX):=\bigvee_{X\in\cX}\suppex(X)$. 

 The following lemma \cite[Lemma 5.1]{Kr2024} has been proved for objects in $\cA$. It can be extended for complex in ${\bf D}^+(\cA)$ as any complex $X\in{\bf D}^+(\cA)$ admits a minimal injective resolution $\cI=0\To I^t\To I^{t+1}\To\dots $ such that $I^i$ is an injective envelope of $Z^i=\Ker(I^i\To I^{i+1})$ for each $i$ (see \cite[Lemma B.1 and Proposition B.2]{Kr2005}.

\medskip
\begin{Lemma}\label{exlem}
Let $X$ be an object of ${\bf D}^+(\cA)$ and $X\to \cI$ be a minimal injective resolution. Then
$\suppex(X)$ equals the localizing subcategory of $\Spec\cA$ consisting of all direct
summands of coproducts of objects in $\{P(I^n)\mid n\in\bbZ\}$.
\end{Lemma}

\medskip
Let $\Inj\cA$ denote the class of injective objects of $\cA$. For every $\Phi$ in ${\bf L}(\Spec\cA)$, we set $$\Inj_{\Phi}\cA:=\{I\in\Inj\cA\mid P(I)\in\Phi\}.$$

\begin{Definition} Let $\Phi\in{\bf L}(\Spec\cA)$.
\begin{enumerate}
\item $\Phi$ is said to be {\it coherent} if each morphism $I^0\To I^1$ in $\Inj_{\Phi}\cA$ fits into an exact sequence $I^0\To I^1\To I^2$ with $I^2\in\Inj_\Phi\cA$.

\item $\Phi$ is said to be {\it specialization-closed} if for each non-zero morphism $I^0\To I^1$ of injective objects in $\cA$ such that $\Im(I^0\To I^1)$ is an essential subobject of $I^1$, the condition $I^0\in\Inj_{\Phi}\cA$ implies that $I^1\in\Inj_{\Phi}\cA$.
\end{enumerate}
\end{Definition}

\medskip 
\begin{Remark}The coherent subcategories are
partially ordered by inclusion and closed under arbitrary intersections. To be more precise, let $(\Phi_\alpha)$ be a family of coherent subcategories in ${\bf L}(\Spec\cA)$ and let $I^0\To I^1$ be a morphism of injectve objects such that $P(I^0),P(I^1)\in\bigcap_\alpha\Phi_\alpha$. If we set $C=\Coker(I^0\To I^1)$ and $I^2=E(C)$, then $P(I^2)\in\bigcap_\alpha\Phi_\alpha$. Then they form a complete lattice. Similarly, specialization-closed subcategories are
partially ordered by inclusion and closed under arbitrary intersections so that they form a complete lattice. 
\end{Remark}

\begin{Definition}\label{thco}
\begin{enumerate}

\item A full subcategory $\cT$ of $\cA$ is said to be {\em thick} provided that for any exact sequence
$$
M_1\to M_2\to M_3\to M_4\to M_5
$$
of objects in $\cA$, if $M_i$ is in $\cT$ for $i=1,2,4,5$, then so is $M_3$.

\item A full subcategory $\cC$ of $\cA$ 
  is said to be \emph{cohomologically stable} if the following conditions hold.
\begin{itemize}

\item $\cC$
is closed under direct summands and any short exact sequence belongs to $\cC$ if two of
its three terms are in $\cC$.

\item for every family $(X_\alpha)$ of objects in $\cC$ the injective
  envelope of  $\coprod_\alpha X_\alpha$ belongs to $\cC$.
\item for every exact sequence $X^0\to X^1\to X^2\to\cdots$ with all
  $X^n$ in $\cC$ the kernel of $X^0\to X^1$ belongs to $\cC$.
\end{itemize}
\end{enumerate}
\end{Definition}
\begin{Remark}
\begin{enumerate}
\item A full subcategory $\cT$ of $\cA$ is thick if and only if it is closed under kernels, cokernels and extensions. We also remark that if 
 a subcategory of $\cA$ is closed under kernels or cokernels, then it is closed under direct summands. In particular, every thick subcategory of $\cA$ is closed under direct summands; (for example see \cite[Remark 2.15]{Ta2009}).
 \item For an exact category $\cA$, a full subactegory $\cC$ of $\cA$ satisfying the first condition of \cref{thco}(2) is called thick by  Krause \cite{Kr2024}. We remark that this definition is different from \cref{thco}(1).
 \end{enumerate}
\end{Remark}

The following lemma extends \cite[Lemma 3.4]{Kr2008} to Grothendieck categories. 
\begin{Lemma}\label{thi}
 Let $\Phi$ be a coherent subcategory of ${\bf L}(\cA)$. Then 
\[\suppex^{-1}(\Phi)=\{M\in\cA\mid \suppex(M)\subseteq \Phi\}\]
is a thick subcategory of $\cA$ closed under injective envelopes.
\end{Lemma} 
\begin{proof}
It is clear by definition that $\suppex^{-1}(\Phi)$ is closed under injective envelopes. By  a simple argument using horseshoe lemma, we can show that $\suppex^{-1}(\Phi)$ is closed under extensions.  Let $f:M\To N$ be a morphism of objects in $\suppex^{-1}(\Phi)$ and let $\cE:0\To E^0\To E^1\To\dots$ and $\cI:0\To I^0\To I^1\To\dots $ be minimal injective resolutions of $M$ and $N$ respectively. Then $f$ can be lifted to a map of complexes $g:\cE\to\cI$. Let
\[\text{con}(g)=0\To E^0\stackrel{\partial^{-1}}\To E^1\oplus I^0\stackrel{\partial^0}\To E^2\oplus I^1\stackrel{\partial^1}\to\dots\] be the mapping cone of $g$ with $Z^i=\Ker\partial^i$ and $B^i=\Im\partial^{i-1}$. This induces an exact triangle $$M\To N\To {\rm con}(g)\To M[-1]$$ which implies that  $\Ker f=H^{-1}({\rm con}(g))=Z^{-1}$ and $\Coker f=H^0({\rm con}(g))$. Since $\Phi$ is coherent and $P(E^0), P(E^1\oplus I^0)$ belong to $\Phi$, $\Ker f$ admits an injective resolution $$0\To\Ker f\To E^0\stackrel{\partial^{-1}}\To E^1\oplus I^0\To J^1\to J^2\dots $$ such that $P(J^i)\in\Phi$ for each $j$ and consequently $\Ker f\in\suppex^{-1}(\Phi)$. By \cite[Lemma 5.2]{Kr2024}, the map $\suppex$ is exact and so the exact sequence $0\To \Ker f\To E^0\To B^0\to 0$ implies that $B^0\in\suppex^{-1}(\Phi)$. On the other hand, $0\to Z^0\To  E^1\oplus I^0\stackrel{\partial^0}\to E^2\oplus I^1\stackrel{\partial^1}\to\dots $ is an injective resolution of $Z^0$ with each $P(E^i\oplus I^{i+1})\in\Phi$ which implies that $Z^0\in\suppex^{-1}(\Phi)$. Therefore the exact sequence $0\To B^0\To Z^0\To \Coker f\To 0$ and the exactness of $\suppex$ implies that $\Coker f\in \suppex^{-1}(\Phi)$.
\end{proof}

\medskip

\begin{Lemma}\label{thick}
Let $\cT$ be a thick subcategory of $\cA$ closed under direct sums and injective envelopes and let $M\in\cA$. Then 
\begin{enumerate}
\item If $\suppex(M)\subseteq \suppex(\cT)$, then $M\in\cT$.

\item $\suppex(\cT)$ is a coherent subcategory of $\Spec\cA$.
\end{enumerate}
\end{Lemma}
\begin{proof}
 (1) Let $0\to M\To E^0\To E^1\To \dots$ be a minimal injective resolution of $M$. The assumption implies that $P(E^i)\in\suppex(\cT)$ for each $i$.  It follows from \cite[Propositon 2.14]{Kr2024} that $P(E^i)=\coprod_\alpha P(E_\alpha)$, where $P(E_\alpha)\in\suppex(X_\alpha)$ for some $X_\alpha\in\cT$. Let $0\To X_\alpha\To E_\alpha^0\To E_\alpha^1\To\dots$  be a minimal injective resolution of $X_\alpha$. Since $\cT$ is thick and  closed under injective envelopes, $E^i_\alpha \in\cT$ for each $i$; and hence \cref{exlem} and \cite[Lemma 2.1]{Kr2024} imply $E_\alpha\in\cT$. Since $\cT$ is closed under direct sums, we conclude that $\coprod_\alpha E_\alpha\in\cT$. Therefore $E^i=E(\coprod E_\alpha)\in\cT$ for each $i$ as $\cT$ is closed under injective envelopes. Since $\cT$ is thick, we deduce that $M\in\cT$.

(2)  Let $\Phi=\suppex(\cT)$ and $I^0\To I^1$  be a morphism in $\Inj_{\Phi}\cA$. For $i=0,1$, since $\suppex(I^i)=\langle P(I^i)\rangle\subseteq\Phi$,  part (1) implies that $I^i\in\cT$. Thus $C=\Coker(I^0\To I^1)\in\cT$ as $\cT$ is thick and $I^2=E(C)\in\cT$ because $\cT$ is closed under injectve envelopes. Therefore $I^0\To I^1$ can be fitted into the exact sequence  $I^0\To I^1\To I^2$ in $\Inj_\Phi\cA$.
\end{proof}
\medskip  

For every $\Phi\in {\bf L}(\Spec\cA)$, set 
\[\cA_\Phi=\suppex^{-1}(\Phi)=\{X\in\cA\mid\suppex(X)\subseteq\Phi\}.\]
The following theorem extends \cite[Proposition 2.17]{Ta2009}.

\begin{Proposition}
Every thick subcategory $\cT$ of $\cA$ closed under direct sums and injective envelopes is cohomologically stable. 
\end{Proposition}
\begin{proof}
It suffices to show that $\cT=\cA_{\suppex(\cT)}$ and so \cite[Lemma 5.4]{Kr2024} implies that $\cT$ is cohomologically stable. Let $M$ be an object in $\cA$  with a minimal injective resolution $0\To M\To E^0\To E^1\To E^2\to\dots $.  If $M\in\cT$, then the assumption on $T$ implies that  $E^i(M)\in\cT$ for each $i$. Hence $P(E^i)\in\suppex(\cT)$ for each $i$ so that $M\in\cA_{\suppex(\cT)}$. For the equality, if $M\in \cA_{\suppex(\cT)}$, then $\suppex(M)\subseteq \suppex(\cT)$. Thus \cref{thick} implies that $M\in\cT$.
 \end{proof}

 The following result extends a theorem due to Krause \cite[Theorem 5.2]{Kr2008} where for a commutative noetherian ring $A$, he provided a classification of coherent subsets of $\Spec A.$

\medskip
\begin{Theorem}\label{krau}
For $\Phi\in{\bf L}(\Spec\cA)$, the following conditions are equivalent.
\begin{enumerate}
\item $\Phi$ is coherent.
\item For every complex $X$ in ${\bf D}^+(\cA)$ we have:
\[\suppex(X)\subseteq \Phi \hspace{0.2cm}\Longleftrightarrow\hspace{0.2cm} \suppex(H^i(X))\subseteq \Phi\text{ for all }i.\]
\item For every complex $X$ in ${\bf D}^+(\cA)$ we have:
\[\suppex(X)\subseteq \Phi \hspace{0.2cm}\Longrightarrow\hspace{0.2cm} \suppex(H^i(X))\subseteq \Phi\text{ for all }i.\]
\end{enumerate}
 \end{Theorem}
 \begin{proof}
 Let $X\in {\bf D}^+(\cA)$. Then $X$ has an injective resolution $\cI=0\To I^t\To I^{t+1}\to \dots $ and using \cite[Proposition B2]{Kr2005} we may assume that $\cI$ is a minimal injective resolution of $M$ with $Z^i=\Ker(I^i\To I^{i+1})$ essential in $I^i$ and $B^i=\Im(I^{i-1}\To I^i)$ for each $i$.
 
 (1)$\Rightarrow$ (2): The assumption on $X$ and \cref{exlem} imply that $P(I^i)\in\Phi$ for all $i$ by so that $Z^i\in\suppex^{-1}(\Phi)$ for all $i$ as  $\suppex^{-1}(\Phi)$ is thick by \cref{thi}. Moreover, the exactness of $\suppex$ and the exact sequence $0\To Z^i\To I^i\To B^{i+1}\To 0$ implies that $B^i\in\suppex^{-1}(\Phi)$ for each $i$. Again the thickness of $\suppex^{-1}(\Phi)$ implies that $H^i(X)\in \suppex^{-1}(\Phi)$ for all $i$. Conversely, if $H^i(X)\in \suppex^{-1}(\Phi)$  for all $i$, then  $Z^t=H^t(X)\in\suppex^{-1}(\Phi)$. Since $Z^t$ is essential subobject of $I^t$, we have $P(Z^t)=P(I^t)$ so that $I^t\in \suppex^{-1}(\Phi)$. Since $\suppex^{-1}(\Phi)$ is thick, the exact sequences  $0\To Z^t\to I^t\To B^{t+1}\To 0$ and $0\To B^{t+1}\To Z^{t+1}\To H^{t+1}(X)\To 0$ conclude that $Z^{t+1}\in\suppex^{-1}(\Phi)$. Now, proceeding by induction $Z^i\in \suppex^{-1}(\Phi)$ for all $i$ so that $P(I^i)\in\Phi$ for all $i$.
 
 (2)$\Rightarrow$ (3) is clear.
 
 (3)$\Rightarrow$ (1): Let $I^0\To I^1$ be a morphism of injective objects with $I^0, I^1\in\Inj_{\Phi}\cA$. Then for  complex $0\To I^0\To I^1\To $, we have $\suppex(\cI)\subset \Phi$. Hence the assumption implies that for $H^1(\cI)=\Coker (I^0\To I^1)=C$, we have $\suppex(C)\subseteq \Phi$. Then if we consider $I^2=E(C)$, we have an exact sequence $I^0\To I^1\To I^2$ such that $I^2\in\Inj_{\Phi}\cA$. 
 \end{proof}
 
 \medskip
 
 \begin{Definition}
\begin{enumerate}
\item A subcategory $\cX$ of ${\bf D}^+(\cA)$ is {triangulated} provided that for every exact triangle $X\To T\To Z\To X[1]$ if two of $X$,$Y$ and $Z$ are in $\cX$, then so is the third.

\item A triangulated subcategory $\cL$ of ${\bf D}^+(\cA)$ is {\it localizing} if it is closed under arbitrary direct sums.  


\item A subcategory $\cX$ of ${\bf D}^+(\cA)$ is {\it closed under homology} if $H^i(X)\in\cX$ for all $X\in\cX$ and all $i\in\mathbb{Z}$.
\end{enumerate}
\end{Definition}

\medskip
\begin{Theorem}\label{loc}
 The assignment $\cL\longmapsto \suppex(\cL\cap \cA)$ induces an inclusion preserving between   \[\{\text{localizing subcategories }\cL\text{ of }{\bf D}^{+}(\cA) \text{ such that }\cL \text{ is closed under homology and }\]\[\cL\cap\cA \text{ is closed under injective envelopes} \}\]\[\longiso\{\text{ coherent subcategories  of }\Spec\cA\}.\]
The inverse map is given by $\Phi\longmapsto \suppex^{-1}(\Phi):=\{X\in{\bf D}^+(\cA)\mid \suppex(X)\subseteq \Phi\}$.
\end{Theorem}
\begin{proof}
It suffices to show that $\cL\cap \cA$ is a thick subcategory of $\cA$ closed under direct sums and so \cref{thick} implies that $\suppex(\cL\cap \cA)$ is coherent in  ${\bf L}(\Spec\cA)$ (we notice that $\cL\cap\cA$ is closed under injective envelopes by the assumption). Since $\cL$ is localizing, $\cL\cap\cA$ is closed under direct sums. Every short exact sequence $0\To M\To N\To K\To 0$ in $\cL\cap \cA$ can be fitted into an exact triangle $M\To N\To K\To M[1]$ in $\cL$. This implies that $\cL\cap\cA$ is closed under extensions. Each morphism $M\stackrel{f}\To N$ in $\cL\cap\cA$ can be fitted into $M\To N\To {\rm con}(f)\To M[1]$. Since $\cL$ is triangulated, con$(f)\in\cL$ so that $H^{-1}(f)=\Ker f$ and $H^0(f)=\Coker f$ belong to $\cL\cap \cA$ as $\cL$ is closed under homology. On the other hand, it follows from \cref{krau} that $\suppex^{-1}(\Phi)$ is closed under homologoy. Moreover, for each $M\in\suppex^{-1}(\Phi)\cap\cA$, we have $\suppex(E(M))\subseteq\suppex (M)$ so that $E(M)\in\suppex^{-1}(\Phi)$. Hence $\suppex^{-1}(-)$ is well defined. If $X\in\cL$, then $H^i(X)\in\cL\cap\cA$ for all $i$ by the definition. Then $\suppex(H^i(X))\subseteq\suppex(\cL\cap\cA)$ for all $i$ and so \cref{krau} implies that $X\in\suppex^{-1}(\suppex(\cL\cap\cA))$. Therefore, $\cL\subseteq \suppex^{-1}(\suppex(\cL\cap\cA)).$ The other side follows from \cref{thick} and consequently 
$$\cL=\suppex^{-1}(\suppex(\cL\cap\cA)).$$
It is clear that $\Phi\subseteq \suppex(\suppex^{-1}(\Phi)\cap\cA)$. Let $P(I)\in\suppex(\suppex^{-1}(\Phi)\cap\cA)$ for some injective object $I$ in $\cA$. Then $\suppex(I)\subseteq \suppex(\suppex^{-1}(\Phi)\cap\cA)$. It follows from \cref{thi} that $\suppex^{-1}(\Phi)\cap\cA$  is thick and hence  \cref{thick} forces that $P(I)\in\Phi$; consequently 
$$\Phi=\suppex(\suppex^{-1}(\Phi)\cap\cA).$$
\end{proof}

For every  subcategory $\cL$ of $\cA$, set 
\[
\cL^{\bot}=\{M\in\cA\mid \Hom(L,M)=0 \text{ for all } \text{ all } L\in\cL\}
;\]\[^\bot\cL=\{M\in\cA\mid \Hom(M,L)=0 \text{ for all } \text{ all } L\in\cL\}.\]

\begin{Definition}
Following \cite[Chap, VI, p. 138]{St}, a {\it torsion theory} of $\cA$ is a pair $(\cT,\cF)$ of classes of objects in $\cA$ such that 
 $\Hom(T,F)=0$ for all $T\in\cT$ and  $F\in\cF$; moreover $\cT^\bot=\cF$ and
 $^\bot \cF=\cT$.
The {\it right prependicular} of a subcategory $\cL$ of $\cA$ is defined as 
\[
\cL^{\bot_{\leq 1}}=\{M\in\cA\mid \Ext^i(L,M)=0 \text{ for all } i\leq 1 \text{ all } L\in\cL.\}
\]The right prependicular subcategories has been extensively studied by Geigle and Lenzing \cite{GL}. If $\cL$ is a localizing subcategory of $\cA$, every object in $\cL^{\bot_{\leq 1}}$ is called $\cL$-{\it closed}. We also set 
\[
\cL^{\bot_{\geq 0}}=\{M\in\cA\mid \RHom(\cL,M)=0\}=\{M\in\cA\mid \Ext^i(L,M)=0 \text{ for all } i\geq 0 \text{ and } L\in\cL\}.
\]
\end{Definition}
\medskip
\begin{Lemma}\label{s1}
Let $\cL$ be a localizing subcategory of $\cA$. Then $\cL^{\bot_{\leq 1}}$ is closed under kernels.
\end{Lemma}
\begin{proof}
 Let $M\stackrel{f}\To N$ be a morphism of objects in $\cL^{\bot_{\leq 1}}$ with $K=\Ker f$, $I=\Im\ f$ and $C=\Coker f$. Applying $\Hom(\cL,-)$ to the exact sequences
\[ 0\To K\to M\To I\To 0\qquad\text{and}\qquad 0\To I\To N\To C\To 0\] we conclude that $K\in\cL^{\bot_{\leq 1}}$.
\end{proof}

\medskip
\begin{Theorem}\label{charc}
Let $\Phi$ be a localizing  subcategory of $\Spec\cA$ such that $P^{-1}(\Phi)$ is closed under products and let $\cL=^\bot(P^{-1}(\Phi))$. Then the following conditions are equivalent.
\begin{enumerate}
\item  $\Phi$ is a coherent subcategory of ${\bf L}(\Spec\cA)$.

 \item $\cL^{\bot_{\leq 1}}=\cL^{\bot_{\geq 0}}$.
 \item $\cL^{\bot_{\geq 0}}$ is a thick subcategory of $\cA$.
 \end{enumerate}
 \end{Theorem}
\begin{proof}
Since $P^{-1}(\Phi)$ is essentially closed by \cite[Theorem 3.3]{Kr2024} and closed under products, it follows from \cite[Lemma 3.1]{Kr2024} that $(\cL,P^{-1}(\Phi))$ a hereditary torsion theory. This implies that $\cL^\bot=P^{-1}(\Phi)$ and we notice that $\cL^\bot$ is closed under injective envelopes by \cite[Chap VI, Proposition 3.2]{St}.

(1)$\Rightarrow$ (2): It is clear that $\cL^{\bot_{\geq 0}}\subseteq\cL^{\bot_{\leq 1}}$. For other inclusion, let $M\in\cL^{\bot_{\leq 1}}$ and let $0\To M\To I^0\To I^1\to\dots $ be a  minimal injective resolution of $M$. Then  $Z^1=\Ker(I^1\to I^2)$ is in $\cL^\bot$. Since $\cL^{\bot}$ is essentially closed, $I^0, I^1\in\cL^\bot$. This implies that $P(I^0),P(I^1) \in\Phi$ by \cite[Theorem 3.3]{Kr2024}. Now, since $\Phi$ is coherent, there exists an injective object $E^2\in\Inj_\Phi\cA$ such that $0\To M\To I^0\To I^1\To E^2$ is exact and $\Ker(I^1\To E^2)\in\cL^{\bot_{\leq 1}}$. Proceeding this way we have an injective resolution \[0\To M\To I^0\To I^1\To E^2\to E^3\to \dots\]
such that $P(E^i)\in \Phi$ for each $i$. As $I^i$ is a direct summand of $E^i$ for each $i$, we conclude that $I^i\in\Inj_\Phi\cA$ so that $I^i\in P^{-1}(\Phi)=\cL^\bot$ for each $i$. Hence $M\in\cL^{\bot_{\geq 0}}$.

 (2)$\Rightarrow$(3): Let $M\stackrel{f}\To N$ be a morphism of objects in $\cL^{\bot_{\geq 0}}$ with $K=\Ker f$, $I=\Im\ f$ and $C=\Coker f$. It follows from \cref{s1} that $K\in\cL^{\bot_{\geq 0}}$ and so a standard argument implies that $I,C\in \cL^{\bot_{\geq 0}}$. Moreover, it is clear that $\cL^{\bot_{\geq 0}}$ is closed under extensions. 
 
 (3)$\Rightarrow$(1): Let $I^0\To I^1$ be a morphism in $\Inj_\Phi\cA$. Then $\Hom(\cL,I^0)=\Hom(\cL,I^1)=0$; and hence $I^0,I^1\in\cL^{\bot_{\geq 0}}$. Since
 $\cL^{\bot_{\geq 0}}$ is thick, $C=\Coker(I^0\to I^1)\in \cL^{\bot_{\geq 0}}$ and we can consider $I^2=E(C)$.
\end{proof}

\medskip

\begin{Proposition}
Every specialization-closed subcategory of $\Spec \cA$ is coherent. 
\end{Proposition}
\begin{proof}
Let $I^0\To I^1$ be a non-zero morphism of injective objects such that $I^0,I^1\in\Inj_\Phi\cA$. Set $C=\Coker(I^0\To I^1)$ and $I^2=E(C)$. Then since $\Phi$ is specialization-closed, we have $I^2\in\Inj_\Phi\cA$.
\end{proof}

For an object $X\in\cA$, Krause \cite{Kr2024} defined the \emph{support} of $X$ as 
\[
  \Supp(X):=\inf\{\cU\in{\bf L}(\Spec\cA)\mid P(X)\in \cU\}=\langle P(X)\rangle.
\]
For a class $\cX\subseteq\cA$, set
\[\Supp(\cX):=\bigvee_{X\in\cX}\Supp(X)=\bigvee_{X\in\cX}\langle P(X)\rangle=\langle
    P(\cX)\rangle.\]
    
 A localizing subcategory is called
\emph{stable} if it is closed under essential extensions, equivalently
if it is closed under injective envelopes. It is clear that stable localizing subcategories are essentially closed. Gabriel \cite{G1962} showed that over a commutative noetherian ring $A$, there is an inclusion-preserving bijection between the set of localizing subcategories of $\Mod A$ and the set of subsets of $\Spec A$ closed under specialization. The following theorem extends this theorem for Grothendieck categories.
\medskip

\begin{Theorem}\label{stablesub}
 The assignment $\cC\longmapsto \Supp(\cC)=P(\cC)$ induces a lattice isomorphism 
\[\{\text{stable localizing subcategories of } \cA\}\longiso\{\text{specialization-closed 
subcategories of } \Spec\cA\}.\]
The inverse map is given by $\Phi\longmapsto P^{-1}(\Phi)$.
\end{Theorem}
\begin{proof}
Let $\cC$ be a stable localizing subcategory of $\cA$ and $\Phi$ be a specialization-closed subset of $\Spec\cA$. We show that $P(\cC)$ is specialization-closed. Let $I^0\To I^1$ be a non-zero morphism of injectives module such that $\Im(I^0\To I^1)$ is essential in $I^1$ and $P(I^0)\in P(\cC)$. Since $\cC$ is essentially closed, it follows from \cite[Theorem 3.3]{Kr2024} that $I^0\in\cC$ and so $\Im(I^0\To I^1)$ is contained in $t(I^1)$, where $t(I^1)$ is the largest subobject if $I^1$ belonging to $\cC$. But this implies that $t(I^1)$ is an essential subobject of $I^1$; and hence $E(t(I^1))=I^1$. Now, since $\cC$ is stable localizing, we conclude that $I^1\in\cC$.

 We now show  $P^{-1}(\Phi)$ is a stable localizing  subcategory of $\cA$. According to \cite[Theorem 3.3]{Kr2024}, $P^{-1}(\Phi)$ is essentially closed; and hence by \cite[Lemma 5.11]{Kr2024} it suffices to show that $P^{-1}(\Phi)=\cA_{\Phi}$. Let $M\in P^{-1}(\Phi)$ and $0\To M\To E^0\To E^1\To\dots$ be a minimal injective resolution of $M$ such that $Z^i=\Ker(E^i\To E^{i+1})$ for each $i$. Since $\Phi$ is specialization-closed, $Z^i$ is an essential subobject of $E^i$ and $P(E^0)\in \Phi$, we conclude that $P(E^i)\in\Phi$ for each $i$. This implies that $M\in\cA_{\Phi}$. The inclusion  $\cA_{\Phi}\subset P^{-1}(\Phi)$ is clear by the definition. Since $\cC$ is essentially closed, it follows from \cite[Theorem 3.3]{Kr2024} that  $P^{-1}(P(\cC))=\cC$ and 
  $P(P^{-1}(\Phi))=\Phi$. 
  \end{proof}
  
  \medskip
   The specialization-closed subcategories are partially ordered by inclusion and closed under arbitrary intersections.  Then they form a complete lattice.


\section{In the case of Locally coherent categories}
Throughout this section, $\cA$ is a locally coherent category with $\fp\cA$, the full subcategory of $\cA$ consisting of finitely presented objects. We study the lattice of Serre subcategories of $\fp\cA$ and the lattice of localizing subcategories of finite type of $\cA$. 

\medskip
\begin{Definition}
\begin{enumerate}
\item An object $M$ in $\cA$ is {\it finitely generated} if whenever there are subobjects $M_i\leq M$ for $i\in I$ satisfying $M=\underset{i\in I}\Sigma M_i$, then there is a finite subset $J\subseteq I$ such that $M=\underset{i\in J}\Sigma M_i$.
\item A finitely generated object $Y$ in $\cA$ is {\it finitely presented} if every epimorphism $f:X\To Y$ in $\cA$ with $X$ finitely generated has a finitely generated kernel $\Ker f$.

\item A finitely presented object in $\cA$ is {\it coherent} if every its finitely generated subobject is finitely presented.

 We denote by fg-$\cA$, fp-$\cA$ and coh-$\cA$, the full subcategories of $\cA$ consisting of finitely generated, finitely presented and coherent objects, respectively. 

\item We recall that a Grothendieck category $\cA$ is {\it locally coherent} if every object in $\cA$ is a direct limit of coherent objects; or equivalently finitely generated subobjects of finitely presented objects are finitely presented. According to \cite[2]{Ro} and \cite{He1997} a Grothendieck category $\cA$ is locally coherent if and only if fp-$\cA=$coh-$\cA$ is an abelian category. 

\item A localizing subcategory $\cX$ of $\cA$ is said to be of {\it finite type} provide that the corresponding right adjoint  of the inclusion functor $\cX\To\cA$ commutes with direct limits; equivalently if the corresponding torsion-free subcategory is closed under direct limits. We denote the collection of localizing subcategories of $\cA$ of finite type by $\Lf(\cA)$. 
 \end{enumerate}
 \end{Definition}
 
 For a full subcategory $\cX$ of $\fp\cA$, we denote by $\overrightarrow{\cX}$ the full subcategory of $\cA$ consisting of direct limits $\underset{\rightarrow}{\rm lim}X_i$ with $X_i\in\cX$ for all $i$. According to \cite[Corollary 2.10]{Kr1997} for every  localizing subcategory $\cX$ of  finite type of $\cA$, we  have $\cX=\overrightarrow{\cX\cap\fp\cA}$.
\medskip

\begin{Proposition}\label{subla}
$\Lf(\cA)$ is a complete lattice. In particular,  \[\bigwedge_\alpha\cU_{\alpha}=\overrightarrow{(\bigcap_\alpha\cU_\alpha)\cap\fp\cA}\] for any family $(\cU_\alpha)$ in $\Lf(\cA)$.
\end{Proposition}
\begin{proof}
 We show that $\Lf(\cA)$ has arbitrary joins. Let $(\cU_\alpha)$ be a family in $\Lf(\cA)$ and $\cU=\bigvee_\alpha\cU_\alpha$ be join in $\bf L\cA$. Then, by \cref{krlem}, we have $$t_\cU(\underset{\rightarrow}{\rm lim}X_i)=\sum_{\alpha}t_{\cU_{\alpha}}(\underset{\rightarrow}{\rm lim}X_i)=\underset{\rightarrow}{\rm lim}\sum_{\alpha}t_{\cU_{\alpha}}(X_i)=\underset{\rightarrow}{\rm lim}t_{\cU}(X_i)$$ which implies that $\bigvee_\alpha\cU_\alpha\in\Lf(\cA).$
It follows from \cite[Chap. III, Proposition 1.2]{St} that $\Lf(\cA)$ is a complete lattice. The equality is clear. 
\end{proof} 
\medskip
\begin{Remark}\label{remm}
We observe that ${\Lf}(\cA)$ is a sublattice of ${\bf L}(\cA)$. To this end, we show that it is closed under finite intersections. Given $\cU,\cV\in\Lf(\cA)$ and $X\in\cU\cap\cV$, we have $X=\underset{\rightarrow}{\rm lim}X_i=\underset{\rightarrow}{\rm lim}Y_i$ where $X_i\in\cU\cap \fp\cA$ and $Y_i\in\cV\cap \fp\cA$ for any $i,j$. Hence every canonical morphism $X_i\To\underset{\rightarrow}{\rm lim}X_i$ factors through $f_i:X_i\to Y_{\alpha_i}$ for some $Y_{\alpha_j}$. We observe that $Z_i=\Im f_i\in\cU\cap\cV\cap\fp\cA$ and there is a canonical epimorphism $\coprod_i Z_i\twoheadrightarrow X$ so that $X\in\overrightarrow{\cU\cap\cV\cap\fp\cA}$. Consequently, $\cU\cap\cV=\overrightarrow{\cU\cap\cV\cap\fp\cA}\in\Lf(\cA)$. 
\end{Remark}
A subcategory $\cS$ of $\fp\cA$ is {\it Serre} if it is closed under subobjects, quotients and extensions. For $\cX\subset\fp\cA$, the smallest Serre subcategory of $\fp\cA$ containing $\cX$ is denoted by $\sqrt{\cX}$. 

\medskip

\begin{Proposition}\label{compact}
$\Lf(\cA)$ is compactly generated.
\end{Proposition}
\begin{proof}
We show that $\overrightarrow{\sqrt{M}}$ is an compact object of $\Lf(\cA)$ for any finitely presented object $M$. If  $\overrightarrow{\sqrt{M}}=\bigvee_\alpha\cU_\alpha$, by \cref{krlem}, we have 
$M=\sum_\alpha t_{\cU_{\alpha}}(M)$. Since $M$ is finitely presented, there exists a positive integer $n$ such that $M=\sum_{i=1}^n t_{\cU_{\alpha_i}}(M)$. This implies that $\overrightarrow{\sqrt{M}}=\bigvee_{i=1}^n\cU_{\alpha_i}$. For any $\cU\in\Lf(\cA)$, we have $\cU=\bigvee_{M\in\cU\cap\fp\cA}\overrightarrow{\sqrt{M}}$; and hence $\Lf(\cA)$ is compactly generated.
\end{proof}

\medskip

\begin{Lemma}\label{serrtwo}
Let $\cS_1,\cS_2$ be Serre subcategories of $\fp\cA$.  

${\rm (i)}$ If $X\in(\overrightarrow{\cS_1}\vee\overrightarrow{\cS_2})\cap\fp\cA$, then
$X=X_1+X_2$ such that $X_i\in\cS_i$;

${\rm (ii)}$ $\sqrt{\cS_1\cup\cS_2}=(\overrightarrow{\cS_1}\vee\overrightarrow{\cS_2})\cap\fp\cA$.
\end{Lemma}
\begin{proof}
To prove (i), it follows from \cref{krlem} that $X=t_{\overrightarrow{\cS_1}}(X)+t_{\overrightarrow{\cS_2}}(X)$. We notice that $t_{\overrightarrow{\cS_1}}(X)=\bigcup_iY_i$ is the direct union of its finitely generated subobjects and hence each $Y_i$ is  finitely presented as it is a subobject of $X$. We now have $X/t_{\overrightarrow{\cS_2}}(X)=\bigcup_i(Y_i+t_{\overrightarrow{\cS_2}}(X))/t_{\overrightarrow{\cS_2}}(X)$ and since  $X/t_{\overrightarrow{\cS_2}}(X)$ is finitely generated, there exits some $i$ such that  $X=Y_i+ t_{\overrightarrow{\cS_2}}(X)$. Again $X/Y_i=\bigcup_i(Y_i+Z_j)/Y_i$, where $t_{\overrightarrow{\cS_2}}(X)=\bigcup_iZ_i$ is the direct union of its finitely generated subobjects and so each $Z_i$ is finitely presented. Since $X/Y_i$ is finitely generated, there exists some $j$ such that $X=Y_i+Z_j$. Since $\overrightarrow{\cS_1}$ and $\overrightarrow{\cS_1}$ are of finite type, it follows from \cite[Corollary 2.10]{Kr1997} that $X_1=Y_i\in\cS_1$ and $X_2=Z_j\in\cS_2$. For (ii), the inclusion $\subseteq$ is clear. If $X\in (\overrightarrow{\cS_1}\vee\overrightarrow{\cS_2})\cap\fp\cA$, then by (i) there exist subobjects $X_i\in\cS_i$ such that $X=X_1+X_2$; and hence $X\in\sqrt{\cS_1\cup\cS_2}$.
\end{proof}

\medskip
Serre subcategories of $\fp\cA$ are partially ordered by inclusion. For Serre subcategories $\cS_1$ and $\cS_2$ of $\fp\cA$, we define $\cS_1\vee\cS_2:=\sqrt{\cS_1\cup\cS_2}=(\overrightarrow{\cS_1}\vee\overrightarrow{\cS_2})\cap\fp\cA$ and $\cS_1\wedge\cS_2:=\cS_1\cap\cS_2$. As it is closed under arbitrary intersections, they form a complete lattice and we denote it by $\bL(\fp\cA)$.
\medskip

\begin{Lemma}\label{four}
\begin{enumerate}

\item If $X\in\bigvee_{\alpha}\cS_\alpha$, then  $X=X_{\alpha_1}+\dots+X_{\alpha_n}$ such that   $X_{\alpha_i}\in\cS_{\alpha_i}$ for each $i$;

\item  If $\{\cU_\alpha\}$  is a family in $\Lf(\cA)$, then $(\bigvee_\alpha\cU_\alpha)\cap\fp\cA=\bigvee_\alpha
(\cU_\alpha\cap\fp\cA)$;

\item  Let $\{\cS_\alpha\}$ be a family in ${\bf L}(\fp\cA)$. Then 

 ${\rm (i)}$ $\bigvee_\alpha\cS_\alpha=(\bigvee_\alpha
\overrightarrow{\cS_\alpha})\cap\fp\cA$
and $\bigcap_\alpha\cS_\alpha=(\bigwedge\overrightarrow{\cS_\alpha})\cap\fp\cA;$

 ${\rm (ii)}$  $\overrightarrow{\bigcap_\alpha\cS_\alpha}=\bigwedge\overrightarrow{\cS_\alpha}$  and $\overrightarrow{\bigvee\cS_\alpha}=\bigvee\overrightarrow{\cS_\alpha}$.
\end{enumerate}
\end{Lemma}
\begin{proof}
To prove (1), it follows from \cref{krlem} that $X=\sum_{\alpha}t_{\overrightarrow{\cS_{\alpha}}}(X)$. Since  $X$ is finitely presented, there exists a positive integer $n$ such that $X=\sum_{i=1}^nt_{\overrightarrow{\cS_{\alpha_i}}}(X)$. Hence $X\in \bigvee_{i=1}^n\overrightarrow{\cS_{\alpha_i}}$ and so by \cref{serrtwo}  there exist $X_i\in\cS_{\alpha_i}$ such that $X=X_{\alpha_1}+\dots+X_{\alpha_n}$. 

(2): The inclusion $\bigvee_\alpha
(\cU_\alpha\cap\fp\cA)\subseteq (\bigvee_\alpha\cU_\alpha)\cap\fp\cA$ is clear. For the converse, since each $\cU_{\alpha}$ is of finite type, if we set $\cS_\alpha=\cU_\alpha\cap \fp\cA$, then $\cU_{\alpha}=\overrightarrow{\cS_\alpha}$. If $X\in(\bigvee_\alpha\cU_\alpha)\cap\fp\cA$, using again \cref{krlem}, we have $X=\sum_{\alpha}t_{\cU_{\alpha}}(X)$. Then $X=\sum_{i=1}^nt_{\cU_{\alpha_i}}(X)$ as $X$ is finitely presented. Thus $X\in(\bigvee_{i=1}^n\overrightarrow{\cS_{\alpha_i}})\cap \fp\cA$ and so \cref{serrtwo} implies that $X=X_1+\dots+X_n$ such that $X_i\in\cS_{\alpha_i}$ for each $i$. Therefore $X\in\bigvee_\alpha\cS_{\alpha}$.

(3): If we set $\cS_{\alpha}=\overrightarrow{\cS_\alpha}\cap\fp\cA$, then the  first equality of (i) follows from (2) and the second follows from \cref{subla}. (ii) follows from (i) easily.
\end{proof}
\medskip

\begin{Theorem}\label{iserre}
Let $\cA$ be a locally coherent category. Then the assignment 
$\cS\longmapsto\overrightarrow{\cS}$ induces a lattice isomorphism 
\[\bL(\fp\cA)\longiso \Lf(\cA).\]
The inverse map is given by $\cL\longmapsto \cL\cap\fp\cA.$ Moreover, the maps preserve arbitrary joins and meets.
\end{Theorem}
\begin{proof}
By virtue of \cite[Corollary 2.10]{Kr1997}, we have $\overrightarrow{\cL\cap\fp\cA}=\cL$ and $\overrightarrow{\cS}\cap\fp\cA=\cS$. The facts and assertions follow from \cref{four}.
\end{proof}

It follows from \cite[Proposition 2.3]{Kr2024} that the complete lattice ${\bf L}(\cA)$ is farme which means that: for an element $\cU$ and a family $(\cV_\alpha)$ in
  ${\bf L}(\cA)$ we have the following equality
  \[\cU\wedge\left(\bigvee_\alpha\cV_\alpha
  \right)=\bigvee_\alpha(\cU\wedge\cV_\alpha).\]
  
\medskip

\begin{Corollary}
Let $\cA$ be a locally coherent category. Then 
 ${\bf L}(\fp\cA)$ is frame and  compactly generated.
\end{Corollary}
\begin{proof}
By \cref{subla} and \cref{remm}, $\Lf(\cA)$ is a sublattice of ${\bf L}(\cA)$, it follows from \cite[Proposition 2.3]{Kr2024} that $\Lf(\cA)$ is frame. Hence \cref{iserre} implies that ${\bf L}(\fp\cA)$ is frame. The second assertion follows from \cref{compact} and \cref{iserre}.
\end{proof}

\medskip
For a locally coherent category $\cA$, Herzog \cite{He1997} defined  the Ziegler spectrum of $\cA$ which induces  a topology on $\Sp\cA$. This topology was originally defined by Ziegler \cite{Zi} associating to a ring $A$, a topological space whose points are the isomorphism classes of indecomposable pure-injective $A$-modules. \begin{Definition}
 For an object $M$ in $\cA$, set  $$\cO(M)=\{I\in\Sp\cA|\hspace{0.1cm} \Hom(M,I)\neq 0\}.$$  
\end{Definition}
For any subcategory $\cX$ of $\fp\cA$, set $\cO(\cX)=\bigcup_{M\in\cX}\cO(M).$ 
Herzog \cite{He1997} proved that the collection $\{\cO(M)|\hspace{0.1cm} M\in\fp\cA\}$  satisfies the axioms for a basis of open subsets defining  the {\it Ziegler topology} of $\cA$.

For $M\in\fp\cA$, we define {\it the Ziegler support of} $M$  which is a localizing subcategory in ${\bf L}(\Spec\cA)$ defined as $\SSupp (M)= \langle P(\cO(M)\rangle.$ For an arbitrary subcategory $\cC$ of $\fp\cA$, we set $\SSupp(\cC)=\langle P(\cO(\cC)\rangle.$
For every $\cU\in\bL(\dSpec \cA)$, set $$\SSupp^{-1}(\cU)=\{M\in\fp\cA\mid \SSupp(M)\subseteq \cU\}.$$ 

The following lemma, part (1) shows that $\SSupp^{-1}(\cU)\in{\bf L}(\fp\cA)$. Set
$${\bf L}(\dSpec\cA)=\{\SSupp(\cC)|\hspace{0.1cm}\cC\hspace{0.1cm}  \textnormal{is a Serre subcategory of}\hspace{0.1cm} \fp\cA\}.$$  We notice that ${\bf L}(\dSpec\cA)$ is  partially ordered by inclusion and so it is a lattice in which the joins coincide with those in ${\bf L}(\Spec\cA)$.
\medskip
\begin{Lemma}\label{prp}
\begin{enumerate}
\item If $0\To N\To M\To K\To 0$ is an exact sequence in $\fp\cA$, then $$\SSupp(M)=\SSupp(N)\vee\SSupp(K).$$
In particular $\SSupp^{-1}(\cU)$ is a Serre subcategory of $\fp\cA$ for every $\cU\in{\bf L}(\dSpec\cA)$.
\item For any subcategory $\cC$ of $\fp\cA$:

${\rm (i)}$  $\SSupp(\cC)=\SSupp(\sqrt{\cC})$;

${\rm (ii)}$  $\SSupp(\cC)=\bigvee_{M\in\cC}\SSupp(M);$

${\rm (iii)}$  $\SSupp^{-1}(\SSupp(\cC))=\sqrt{\cC}$ for $\cC\subset\fp\cA$.\\
For $\cS_1,\cS_2\in{\bf L}(\fp\cA)$ and $\cU_1,\cU_2\in{\bf L}(\dSpec\cA)$, we have:
\item $\SSupp(\cS_1\vee\cS_2)=\SSupp(\cS_1)\vee\SSupp(\cS_2)$.
\item $\ZSupp^{-1}(\cU_1\vee\cU_2)=\ZSupp^{-1}(\cU_1)\vee\ZSupp^{-1}(\cU_1).$ 
\item  $\SSupp(\cS_1\cap\cS_2)=\SSupp(\cS_1)\cap\SSupp(\cS_2)$.
\item $\ZSupp^{-1}(\cU_1\cap\cU_2)=\ZSupp^{-1}(\cU_1)\cap\ZSupp^{-1}(\cU_1).$
\end{enumerate}
\end{Lemma}
\begin{proof}
(1) follows from the fact that $\cO(M)=\cO(N)\cup\cO(K)$. The second assertion is clear.

 (2): (i) follows as $\cO(\cC)=\cO(\sqrt{\cC})$ by \cite[Propositions 3.2 and 3.3]{He1997} and $\cO(\cC)=\bigcup_{M\in\cC}\cO(M)$. (ii) Let $P(E)\in \SSupp(\cC)$. We may assume that $P(E)$ is simple. Then there exists $M\in\cC$ such that $P(E)\in\SSupp(M)$. The converse is clear.
(iii): Let $M\in\SSupp^{-1}(\SSupp(\cC))$. Then $\SSupp(M)\subseteq \SSupp(\cC)$ and so $\cO(M)\subseteq\cO(\cC)$. It follows from \cite[Theorem 3.8]{He1997} that $M\in\sqrt{\cC}$. The converse is clear. 

(3): $\SSupp(\cS_1\vee\cS_2)=\SSupp(\cS_1\cup\cS_2)=\SSupp(\cS_1)\vee\SSupp(\cS_2)$ where the last equality follows from (2)(ii).

(4): Given $M\in\ZSupp^{-1}(\cU_1\vee\cU_2)$, we have $\SSupp(M)\subseteq \cU_1\vee\cU_2$.  By the definition, there exist $\cS_1,\cS_2\in{\bf L}(\fp\cA)$ such that $\cU_i=\SSupp(\cS_i)$ for $i=1,2$. Then (3) implies that $\cO(M)\subseteq \cO(\cS_1\vee\cS_2)$, so that $M\in\cS_1\vee\cS_2$. According to \cref{four} there exists $M_i\in\cS_i$ such that $M=M_1+M_2$. Thus $M_i\in\SSupp^{-1}(\cU_i)$ and  so $M\in\ZSupp^{-1}(\cU_1)\vee\ZSupp^{-1}(\cU_2)$ The converse is clear.

 (5): Let $P(E)\in\SSupp(\cS_1)\wedge\SSupp(\cS_2)$. We may assume that $P(E)$ is simple and so by the proof of \cite[Theorem 3.4]{He1997}, we have $E\in\cO(\cS_1\cap\cS_2)$. This implies that $P(E)\in\ZSupp(\cS_1\cap\cS_2)$. The converse is clear. (6) is clear.
 \end{proof}

\begin{Remark}
We observe  that ${\bf L}(\dSpec\cA)$ is a complete lattice. More precisely. if  $(\SSupp(\cC_{\alpha}))$ is a family of localizing subcatgories in $\bL(\dSpec\cA)$, it follows from \cref{prp} that $\bigvee\SSupp_\alpha(\cC_{\alpha})=\SSupp(\sqrt{\bigcup_\alpha\cC_{\alpha}})\in\bL(\dSpec\cA)$. Therefore \cite[Chap III, Proposition 1.2]{St} implies that ${\bf L}(\dSpec\cA)$ is a complete lattice
\end{Remark}
\medskip

Our next result shows that the Ziegler support in ${\bf L}(\dSpec\cA)$ and the support of localizing subcategories of finite type of $\cA$ defined by Krause \cite{Kr2024} are closely related to open subsets in the Ziegler topology of $\cA$.
\medskip
\begin{Proposition}\label{lft}
Let $\cC$ be a localizing subcategory of finite type of $\cA$. Then 
\begin{enumerate}
\item $\Supp(\cC)_d=\SSupp (\cC\cap\fp\cA)$.
\item $P^{-1}(\Supp(\cC))\cap \Sp\cA=\cO(\cC\cap\fp\cA)$.
\end{enumerate}
\end{Proposition}
\begin{proof}
(1): Let $P(E)\in\Supp(\cC)_d$. We may assume that $E$ is an indecomposable injective object and so $P(E)$ is a direct summand of $P(E(M))$ for an object $M\in\cC$. Then by \cite[Lemma 2.1]{Kr2024}, there exists an subobject $X$ of $M$ such that $E(X)=E$; and hence $\Hom(M,E)\neq 0$. Hence there exists a finitely presented object $N$ of $\cC$ such that $\Hom(N,E)\neq 0$ as $\cC$ is of finite type. Conversely, if $P(E)\in \langle P(\cO(\cC\cap\fp\cA))\rangle$ for some indecomposable injective object $E\in\cO(\cC\cap\fp\cA)$, then there exists a finitely presented object $M\in\cC$ such that $\Hom(M,E)\neq 0$. If $f:M\To E$ is a non-zero morphism, then $\Im f\in\cC$ and $E=E(\Im f)$; and hence $P(E)\in\Supp_d(\cC)$.

(2): Let $E\in P^{-1}(\Supp(\cC))\cap \Sp\cA$. Since $P(E)$ is simple, by \cite[Proposition 2.14]{Kr2024}, there exists  $M\in\cC$ such that $P(E)\in\Supp(M)=\langle P(M)\rangle$. Then $E$ is a direct summand of $E(M)$ so that $\Hom(M,E)\neq 0$. Since $\cC$ is of finite type, $M=\underset{\rightarrow}{\rm lim}M_i$ where $M_i\in\cC\cap \fp\cA$ for each $i$. Hence there exists some $i$ such that  $\Hom(M_i,E)\neq 0$. This implies that $E\in\cO(M_i)\subseteq \cO(\cC\cap\fp\cA)$. The converse follows from (1) easily. 
\end{proof}

\medskip

Let $\cE(\Sp\cA)$ be the class of essentially closed subcategories of $\cA$ generated by open subsets of $\cO(\cS)$ where $\cS\in{\bf L}(\fp\cA)$. Then $\cE(\Sp\cA)$ is partially ordered by inclusion. For every $\cC_1,\cC_2\in \cE(\Sp\cA)$, there exist $\cS_1,\cS_2\in{\bf L}(\fp\cA)$ such that $\cC_i=\langle\cO(\cS_i)\rangle$ for $i=1,2$. We define $$\cC_1\vee\cC_2:=\langle\cO(\cS_1\vee\cS_2)\rangle=\langle\cO(\cS_1\cup\cS_2)\rangle.$$
We observe that $\cE(\Sp\cA)$ admits arbitrary joins so that it is a complete lattice.
\medskip
\begin{Lemma}\label{prp1}
Let $\cS_1,\cS_2\in{\bf L}(\fp\cA)$ and let $\cC_1=\langle \cO(\cS_1)\rangle$ and $\cC_2=\langle \cO(\cS_2)\rangle$. Then
\begin{enumerate}
\item $\langle\cO(\cS_1\cap\cS_2)\rangle= \langle\cO(\cS_1)\rangle\cap \langle\cO(\cS_2)\rangle$.

\item $\langle\cO(\cS_1\vee\cS_2)\rangle= \langle\cO(\cS_1)\rangle\vee \langle\cO(\cS_2)\rangle$.
\item $P(\cC_1\cap\cC_2)=P(\cC_1)\cap P(\cC_2)$.

\item $P(\cC_1\vee\cC_2)=P(\cC_1)\vee P(\cC_2)$.
\end{enumerate}
\end{Lemma}
\begin{proof}
The proofs are straightforward using the fact that $\cO(\cS_1\cap\cS_2)=\cO(\cS_1)\cap \cO(\cS_2)$ by \cite[Theorem 3.4]{He1997} and $\cO(\cS_1\vee\cS_2)=\cO(\cS_1)\cup\cO(\cS_2)$.
\end{proof}

\medskip
\begin{Theorem}\label{char1}
There are the lattice isomorphisms:

\[\begin{tikzcd} 
\Lf(\cA)  \arrow[rightarrow,yshift=.75ex,rr,"(-)\cap\fp\cA"]
 	&& \arrow[rightarrow,yshift=-.75ex,ll,"\overrightarrow{(-)}"] 
 	\bL(\fp\cA) \arrow[rightarrow,yshift=.75ex,rr,"\SSupp(-)"]
 	&& \arrow[rightarrow,yshift=-.75ex,ll,"\SSupp^{-1}(-)"] 
 	\bL(\dSpec\cA) \arrow[rightarrow,yshift=.75ex,rr,"P^{-1}(-)"]
 	&& \arrow[rightarrow,yshift=-.75ex,ll,"P(-)"] 
 	\cE(\Sp\cA)
\end{tikzcd}\]

\end{Theorem}
\begin{proof}
\cref{four,prp,prp1} show that all maps are lattices morphisms.

The first isomorphisms follows from [Kr, Corollary 2.10].

 For the second morphisms: It is clear that $\cS\subseteq\SSupp^{-1}(\SSupp(\cS))$. Given $M\in\SSupp^{-1}(\SSupp(\cS))$, we have  $\SSupp(M)\subset\SSupp(\cS)$. Then
 $\cO(M)\subset \cO(\cS)$ and so $M\in\cS$ by \cite[Theorem 3.8]{He1997}. The equality $\cU=\SSupp(\SSupp^{-1}(\cU)$ is straightforward.
 
   To prove the third morphisms, for any $\cU\in\bL(\dSpec\cA)$,  $P^{-1}(\cU)$ is the essentially subcategory of $\cA$ generated by $\cO(\SSupp^{-1}(\cU))$ and for any  essentially subcategory $\langle \cO\rangle$ in $\cE(\Zg\cA)$, we have $P(\langle \cO\rangle)=\SSupp(\cS_{\cO})$, where $\cS_{\cO}=\{M\in\fp\cA|\hspace{0.1cm}\cO(M)\subseteq \cO\}$. These maps provide mutually inverse bijection between $\bL(\dSpec\cA)$ and $\cE(\Zg\cA)$. 
\end{proof}

The following lemma has been mentioned in the context of \cite{Kr2024}. However, we include  it due to its importance in the subsequent results.

\medskip

\begin{Lemma}\label{ort}
Let $\cL$ be a localizing subcategory of finite type of  $\fp\cA$. Then ${\cL}^\bot$ is an essentially closed subcategory.
\end{Lemma}
\begin{proof}
It is clear that $(\cL\cap\fp\cA)^\bot=\overrightarrow{\cL}^\bot$ and so $\overrightarrow{\cL}^\bot$ is closed under coproducts and subobjects. Moreover, since $\cL$ is localizing, $\cL^\bot$ is closed under injective envelopes so that it is essentially closed. 
\end{proof}

\medskip
 For a localizing subcategory $\cL$ of finite  type of $\cA$, although $\cL$ need not be cohomological stable (see \cite[Remark 5.7(2)]{Kr2024}), we have the following fact.

\begin{Proposition}\label{eqs}
Let  $\cL$  be a localizing subcategory of finite type of  $\cA$. Then $\cL^{\bot_{\geq 0}}$ is cohomologically stable. In particular, $\cL^{\bot_{\geq 0}}=\cA_{P(\cL^\bot)}.$
\end{Proposition}
\begin{proof}
In view of \cite[Theorem 5.5]{Kr2024}, it suffices to  prove the equality. Let $M$ be an object in $\cA$ and $0\To M\To I^0\To I^1\To\dots $ be a minimal injective resolution of $M$ with $Z^i=\Ker(I^i\To I^{i+1})$ for each $i>0$. If $M\in\cL^{\bot_{\geq 0}}$, then $I^0\in\cL^{\bot}$ because ${\cL}^\bot$ is essentially closed. The exact sequence $0\To M\To I^0\To Z^1\to 0$ implies that  Hence $Z^1\in\cL^{\bot_{\geq 0}}$. Proceeding by induction, we deduce that $I^i\in\cL^\bot$ for each $i$. Now, \cref{ort} implies that $P(I^i)\in P(\cL^\bot)$ for each $i$ so that $M\in\cA_{P(\cL^\bot)}$. For other inclusion, let $M\in\cA_{P(\cL^\bot)}$. Using again \cref{ort}, we conclude that $I^i\in\cL^\bot$. Consequently, $\Ext^i(\cL,M)=0$ for all $i\geq 0$.
\end{proof}

\begin{Theorem}\label{coch}
The assignment $\cL\longmapsto P(\cL^{\bot})$
induces an inclusion-reversing bijection  
\[{\bf Lf}(\cA)\longrightarrow\{\Phi\in{\bf L}(\Spec \cA)\mid P^{-1}(\Phi)\text{ is closed under products and filtered colimits }\}.\]
The inverse map is given by $\Phi\longmapsto ^\bot(P^{-1}(\Phi))$.
\end{Theorem}
\begin{proof}
It is clear that  $\cL^\bot$ is closed under products. Moreover, since $\cL$ is of finite type, $(\cL\cap\fp\cA)^\bot=\cL^\bot$ is essentially closed by \cref{ort}  and closed under filtered colimits.  Hence the first map is well-defined. Set $\cX=^\bot(P^{-1}(\Phi))$. Since $P^{-1}(\Phi)$ is closed under products, it follows from \cite[Lemma 3.11]{Kr2024} that $(\cX,P^{-1}(\Phi))$ is a torsion theory. Moreover, since $P^{-1}(\Phi)$  is closed under filtered colimits, $\cX$ is a localizing subcategory of finite type and hence the second map is well-defined. Furthermore, it follows from \cite[Theorem 3.3]{Kr2024} that \[P(((^\bot P^{-1}(\Phi))^\bot)=\Phi \quad\text{ and } \quad ^\bot(P^{-1}(P(\cL^\bot)))=\cL.\]
\end{proof}



\begin{Corollary}
 The assignment $\cL\longmapsto P(\cL^{\bot})$
induces inclusion-reversing bijection 
\[\{\cL\in{\Lf}(\cA)\mid \cL^{\bot_{\leq 1}}\text{ is thick}\} \longrightarrow\{\Phi\in{\bf L}(\Spec \cA)\mid  \Phi \text{ is coherent and }\]\[  P^{-1}(\Phi)\text{ is closed under products and  filterred colimits}\}.\]
The inverse map is given by $\Phi\longmapsto ^\bot(P^{-1}(\Phi))$.
\end{Corollary}
\begin{proof}
 It follows from \cref{charc} that $P(\cL^\bot)$ is coherent and  $(^\bot P^{-1}(\Phi))^{\bot_{\leq 1}}$ is thick. Hence in view of \cref{coch}, the maps are is well-defined   and \[P(((^\bot P^{-1}(\Phi))^\bot)=\Phi \quad\text{ and } \quad ^\bot(P^{-1}(P(\cL^\bot)))=\cL.\]
\end{proof}


\section{In the case of locally noetherian categories}

A Grothendieck category $\cA$ is said to be {\it locally noetherian} if it has a small generating set of noetherian objects. Throughout this section $\cA$ is a locally noetherian category.

Krause \cite[Theorem 5.5]{Kr2024} proved that there is a lattice isomorphism between cohomologically stable subcategories of $\cA$  and localizing subcategories of ${\bf L}(\Spec\cA)$.  Over a commutative noetherian ring $A$, an analogous result has already been proved  by Takahashi \cite[Theorem 2.8]{Ta2009} where he introduced $E$-stable subcategories of $\Mod A$ and established a bijection between $E$-stable subcategories closed under direct summand and direct sums and subsets of $\Spec A$. By the definition, cohomologically stable categories are closed under direct summands and the following proposition shows that they cover the other  condition on $E$-stable subcategories if $\cA$ is locally noetherian.
  
\begin{Proposition}\label{dirsum}
If $\Phi\in{\bf L}(\Spec\cA)$, then $\suppex^{-1}(\Phi)$ is closed under direct sums. In particular, every cohomologically stable subcategory of $\cA$ is closed under direct sums.
\end{Proposition}
\begin{proof}
 Let $(M_\alpha)$ be a family of objects in $\suppex^{-1}(\Phi)$ and  
$0\To M_{\alpha}\To I^0_{\alpha}\To I^1_\alpha\to \dots$ be the minimal injective resolution of $M_\alpha$ with $Z^i_\alpha=\Ker (I^i_\alpha\to I^{i+1}_\alpha)$ is essential in $I_\alpha^i$. Hence $\coprod_\alpha Z_\alpha^i$     is essential in $\coprod_\alpha I_\alpha^i$  for each $i$ so that $$0\To\coprod M_{\alpha}\To \coprod I^0_{\alpha}\To \coprod I^1_\alpha\to \dots$$ is a minimal injective resolution of $\coprod M_\alpha$. Since $P$ preserves coproducts, we have $P(\coprod_\alpha I_\alpha^i)\cong \coprod_\alpha P(I_\alpha^i)\in \Phi$ for each $i$. This implies that $\suppex(\coprod_\alpha M_\alpha)\subseteq\Phi$; and consequently $\coprod_\alpha M_\alpha\in\suppex^{-1}(\Phi)$. To prove the second claim, if $\cC$ is a cohomologically stable, it follows from \cite[Theorem 5.5]{Kr2024} that there exists $\cU\in{\bf L}(\cA)$ such that $\cC=\suppex^{-1}(\cU)$. Hence the claim follows by first part.
\end{proof}

 The earlier version of the following theorem has been proved by Krause \cite{Kr2008} for $\Mod A$ where $A$ was a commutative noetherian ring. The following theorem was proved by Wu and Ma \cite[Theorem 3.6]{WM2024}, where they defined coherent subsets in $\Sp\cA$,  the set of isomorphism classes of indecomposable injective objects in $\cA$. We reprove the theorem, where the classification is based on  coherent subcategories of $\Spec\cA$. 
\medskip

\begin{Theorem}
 The assignment $\cT\longmapsto \suppex(\cT)$ induces a lattice isomorphism 
\[\{ \text{thick subcategories of }\cA \text{ closed under direct sums and  injective envelopes} \}\]\[\longiso\{\text{ coherent subcategories  of }\Spec\cA\}.\]
The inverse map is given by $\Phi\longmapsto \suppex^{-1}(\Phi)$.
\end{Theorem}
\begin{proof}
It follows from  \cref{thi} and \cref{dirsum} that $\suppex^{-1}(\Phi)$ is thick a subcategory of $\cA$ closed under direct sums and injective envelopes. On the other hand, $\suppex(\cT)$ is a coherent subset of $\Spec\cA$ by \cref{thick}. Clearly  $\suppex(\suppex^{-1}(\Phi))\subset \Phi$. For other side, if $P(I)\in\Phi$ for an injective object $I$, then $\suppex(I)=\langle P(I)\rangle\subset \Phi$ so that $P(I)\in\suppex(\suppex^{-1}(\Phi))$. Consequently $$\suppex(\suppex^{-1}(\Phi))=\Phi.$$ The equality $\suppex^{-1}(\suppex(\cT))=\cT$ is clear by using \cref{thi}.
\end{proof}

\medskip
Let $A$ be a commnutative noetherian ring and $\Mod A$ be the category of $A$-modules. For every $\Phi\in{\bf L}(\Spec\Mod A)$, set $\Ass(\Phi)=\bigcup_{P(I)\in\Phi}\Ass (I)$. We recall from \cite{Kr2008} that a subset $\Psi$ in $\Spec A$ is {\it coherent} if every morphism $I^0\To I^1$ of injective modules with $\Ass(I^0),\Ass(I^1)\subset\Psi$ can be fitted into an exact srquence of injective modules $I^0\To I^1\to I^2$ such that $\Ass(I^2)\subset\Psi$. The following result shows that coherent subcategories of ${\bf L}(\Spec\Mod A)$ agrees with the coherent subsets of $\Spec A$.
\medskip
\begin{Proposition}
Let $A$ be a noetherian commutative ring.
 The assignment $\Phi\longmapsto \Ass(\Phi)$ induces an inclusion preserving bijection between coherent subcategories of ${\bf L}(\Spec\Mod A)$ and cohernet subsets of $\Spec A$.

 The inverse map is given by $\Psi\longmapsto \langle P(A/\Psi)\rangle$ where $A/\Psi=\{A/\frak p\mid \frak p\in\Psi\}$.
\end{Proposition}
\begin{proof}
We first show that $\Ass(\Phi)$ is a coherent subset of $\Spec A$. Let $I^0\To I^1$ be a morphism of injective modules with $\Ass(I^0),\Ass(I^1)\subset\Ass(\Phi)$. Then $P(I^0),P(I^1)\in\Phi$ and so there exists an injective module $I^2$ such that $P(I^2)\in\Phi$ and $I^0\To I^1\To I^2$ is exact. Therefore $\Ass(I^2)\subseteq \Ass(\Phi)$. Moreover, a standard argument shows that $\{A/\frak p\mid \frak p\in\Psi\}$ is a coherent subcategory of ${\bf L}(\Spec \Mod A)$. 
The equalities $\langle P(A/\Ass(\Phi))\rangle=\Phi$ and $\Ass(\langle P(A/\Psi)\rangle)=\Psi$ are straightforward. 
 
\end{proof}

\medskip
\begin{Proposition}
Let $A$ be a noetherian commutative ring. The assignment $\Phi\longmapsto \Ass(\Phi)$ induces an inclusion preserving bijection between  specialization-closed subcategories of ${\bf L}(\Spec\Mod A)$ and  specialization-closed subsets of $\Spec A$. 

The inverse map is given by $\Psi\longmapsto \langle P(A/\Psi)\rangle$ where $A/\Psi=\{A/\frak p\mid \frak p\in\Psi\}$.
\end{Proposition}
\begin{proof}
We first show that $\Ass(\Phi)$ is a specialization-closed subset of $\Spec A$. Let $\frak p,\frak q$ be prime ideals of $A$ such that $\frak p\in\Ass(\Phi)$ such that $\frak p\subset \frak q$. The canonical epimorphism $A/\frak p\To A/\frak q$ induces a nonzero $A$-homomorphism $E(A/\frak p)\To E(A/\frak q)$. The assumption implies that $P(E(A/\frak p))\in\Phi$. Since $\Phi$ is specialization-closed, $P(E(A/\frak q))\in\Phi$; and consequently $\frak q\in \Ass(\Phi)$.

We now show that $\langle P(A/\Psi)\rangle$  is a specialization-closed subcategory of ${\bf L}(\Spec\Mod A)$. Let $I^0\stackrel{f}\To I^1$ be a non-zero $A$-homomorphism of injective modules such that $\Im f$ is an essential submodule of $I^1$ and $P(I^0)\in\langle P(A/\Psi)\rangle$. Then $I^0=\bigoplus_{\frak p\in\Psi} E(A/\frak p)^{\mu_{\frak p}}$ and moreover $I^1=\bigoplus_{\frak q\in\Ass(I^1)} E(A/\frak q)^{\nu_{\frak q}}$. Since $N$ is an essential submodule of $I^1$ for each $E(A/\frak q)$ occurring  in $I^1$, we have $\frak q\in \Ass(\Im f)$. Then there exist  nonzero elements $x\in I^0$ and $y\in N$ such that $f(x)=y$ and $\Ann(y)=\frak q$. Clearly $\Ann(Ax)\subseteq \frak q$ and so there exists  $\frak p\in \Ass(Ax)$, such that $\frak p\subseteq\frak q$. Now, since $\Psi$ is specialization-closed and $\frak p\in\Psi$, we conclude that $\frak q\in\Psi$; and consequently $P(I^1)\in\langle P(A/\Psi)\rangle$. The equalities $\langle P(A/\Ass(\Phi))\rangle=\Phi$ and $\Ass(\langle P(A/\Psi)\rangle)=\Psi$ are straightforward. 
\end{proof}

\medskip

\begin{Remark}
For a commutative noetherian ring $A$, let $\cL$ be a localizing subcategory of ${\bf D}(A)$
closed under homology. In view of the proof of 
\cref{loc}, $\cL\cap\Mod A$ is a thick subcategory of $\Mod A$. Moreover, it follows from \cite[Lemma 3.5]{Kr2008} that $\cL\cap\Mod A$ is closed under injective envelopes. Then \cref{loc} extends a theorem due to Takahashi \cite[Theorem 2.12]{Ta2009} for arbitrary Grothendieck categories.


\end{Remark}



\end{document}